\documentclass{article}

\usepackage[final]{graphicx} 
\usepackage{subcaption}    
\usepackage{color} 
\usepackage{varioref}  
\usepackage{rotating}  

\usepackage{pdflscape}		
\usepackage{setspace}		
\usepackage{mathrsfs}			
\usepackage{mathtools}
\usepackage{verbatim}			
\usepackage{appendix}

\usepackage[final]{listings} 
\usepackage[pdftex,hidelinks]{hyperref}
\usepackage{cite} 

\usepackage{tikz,ulem} 
\usetikzlibrary{calc}
\usetikzlibrary{patterns}
\usetikzlibrary{positioning}
\usetikzlibrary{decorations.markings}
\usetikzlibrary{decorations.pathmorphing}
\usetikzlibrary{arrows}

\usepackage{latexsym}
\usepackage{amsfonts}
\usepackage{amssymb}
\usepackage{amsmath}
\usepackage{amsthm}
\usepackage{bbm}
\usepackage{afterpage}
\usepackage{fancyhdr}
\usepackage{lipsum} 
\usepackage{hyperref}
\usepackage{amsfonts}
\usepackage{amsthm}
\usepackage{amsmath}
\usepackage{amscd}
\usepackage{amssymb}
\usepackage{amsbsy}
\usepackage{epsf,graphicx,amsmath,verbatim,epsfig,amssymb,bbm}
\usepackage{tikz-cd,tikz, MnSymbol, relsize,enumitem, dsfont}
\usepackage{enumitem}
\usepackage{amsthm}

\usepackage[all]{xy}

\usepackage{ulem}

\newtheorem{theorem}{Theorem}[section]

\newtheorem{lemma}[theorem]{Lemma}
\newtheorem{corollary}[theorem]{Corollary}

\theoremstyle{definition}
\newtheorem{definition}[theorem]{Definition}

\newtheorem{remark}[theorem]{Remark}
\newtheorem{example}[theorem]{Example}

\newtheorem{convention}[theorem]{Convention}

\theoremstyle{remark}

\newcommand{\p}{\varphi}
\newcommand{\RR}{{\rm I\kern -1.6pt{\rm R}}}

\def\ds{\displaystyle}

\def\deq{:=}

\def\R{\mathbb{R}}

\def\S{\mathbb{S}}

\def\id{\mathbbm{1}}

\def\H{\mathbb{H}}

\def\closure{\overline}

\def\span{\text{span}}

\newcommand{\D}{\mathbb{D}}

\newcommand{\der}[2]{D_{#2}{#1}} 
\newcommand{\de}[1]{D{#1}}

\newcommand{\basis}[1]{\frac{\partial}{\partial #1}}

\DeclareMathOperator{\Ima}{Im}

\def\cc{Carnot-Carath\'{e}odory~}

\newcommand{\dcc}[1]{d_{CC}^{#1}}
\def\H{\mathbb{H}}
\newcommand{\hm}[1]{\mathcal{H}^{#1}}
\newcommand{\pilip}[1]{\pi_{#1}^{\text{Lip}}}
\newcommand{\bcc}[1]{B_{CC}^{#1}}
\newcommand{\length}[1]{l^{#1}}

\def\lipmap{{\sf Map}^{\sf Lip}}

\def\RR{\mathbb R}

\newcommand\blfootnote[1]{%
  \begingroup
  \renewcommand\thefootnote{}\footnote{#1}%
  \addtocounter{footnote}{-1}%
  \endgroup
}

\begin{document}

\title{Lipschitz Homotopy Groups of Contact 3-Manifolds}
\author{Daniel Perry}
{\let\newpage\relax\maketitle}
  \begin{abstract}
We study contact 3-manifolds using the techniques of sub-Riemannian geometry and geometric measure theory, in particular establishing properties of their Lipschitz homotopy groups. We prove a biLipschitz version of the Theorem of Darboux: a contact $(2n+1)$-manifold endowed with a sub-Riemannian structure is locally biLipschitz equivalent to the Heisenberg group $\H^n$ with its \cc metric. Then each contact $(2n+1)$-manifold endowed with a sub-Riemannian structure is purely $k$-unrectifiable for $k>n$. We then extend results of Dejarnette et al. \cite{Dej} and Wenger and Young \cite{Weg} on the Lipschitz homotopy groups of $\H^1$ to an arbitrary contact 3-manifold endowed with a \cc metric, namely that for any contact 3-manifold the first Lipschitz homotopy group is uncountably generated and all higher Lipschitz homotopy groups are trivial. Therefore, in the sense of Lipschitz homotopy groups, a contact 3-manifold is a $K(\pi,1)$-space with an uncountably generated first homotopy group. Along the way, we prove that each open distributional embedding between purely 2-unrectifiable sub-Riemannian manifolds induces an injective map on the associated first Lipschitz homotopy groups. Therefore, each open subset of a contact 3-manifold determines an uncountable subgroup of the first Lipschitz homotopy group of the contact 3-manifold. \blfootnote{{\bf Mathematical Reviews subject classification.} Primary: 53C17, 57K33; Secondary: 28A75, 55Q70, 53D10 \\ {\bf Acknowledgments.} The author was supported by NSF awards 1507704 and 1812055. }
  \end{abstract}


\normalem

\section{Introduction.}

In this paper, we use metric geometry to show a sense in which each connected contact 3-manifold is a $K(\pi,1)$-space with an uncountably generated first homotopy group.  After we endow a contact 3-manifold with a metric structure sensitive to the distribution, we probe the space with Lipschitz maps. The metric is the \cc metric of sub-Riemannian geometry. Our results are phrased in terms of Lipschitz homotopy groups. 


In contact topology, the classification of contact 3-manifolds is of active interest. The primary tools used for better understanding contact 3-manifolds have been Reeb vector fields and singular foliations. For an introduction to contact geometry, see\cite{Gei}. For a thorough overview of techniques and results in contact topology, see \cite{Gei2}. For results classifying contact 3-manifolds, see \cite{Eliashberg} or \cite{Massot}.

In this paper, we instead apply the techniques of sub-Riemannian geometry to study contact 3-manifolds. Though contact 3-manifolds do not have an inherent sub-Riemannian structure, they can be endowed with such structure. Once the sub-Riemannian structure is fixed, the underlying manifold inherits a \cc metric structure which is sensitive to the contact distribution. For an overview of sub-Riemannian geometry, see \cite{Mon}.

Our primary tool for studying the metric structure of a contact manifold (of any dimension) is Lipschitz homotopy groups. Dejarnette et al. \cite{Dej}, first introduced Lipschitz homotopy groups in order to study Sobolev mappings into the sub-Riemannian manifold $\H^1$. Since Lipschitz homotopy groups were introduced, they have been calculated for various Heisenberg groups in \cite{Dej}, \cite{Haj}, \cite{HajSch}, \cite{HajTys}, and \cite{Weg}. 

By the Theorem of Darboux, the distributional structure of contact $(2n+1)$-manifolds is locally modeled by the contact structure of the $n$th Heisenberg group $\H^n$ \cite{Darb}. So, we are able to apply the strategies and approaches of sub-Riemannian geometry used to study $\H^n$ to study contact manifolds. Indeed, we prove a biLipschitz version of the Theorem of Darboux which says that the metric structure of a contact $(2n+1)$-manifold (after being endowed with a \cc metric) is locally modeled by the metric space $\H^n$ (Corollary~\ref{biLipschitz Darboux}).

Among the metric properties of $\H^1$, we make use of $\H^1$ being purely 2-unrectifiable in the sense of \cite{Amb}. Indeed, for any $k>n$, the $n$th Heisenberg group $\H^n$ is purely $k$-unrectifiable \cite{Mag04}. Using the biLipschitz version of the Theorem of Darboux, any contact $(2n+1)$-manifold (with \cc metric) is also purely $k$-unrectifiable for $k>n$ (Theorem~\ref{contact 3-mflds purely unrectifiable}).

Once shown that contact 3-manifolds are purely 2-unrectifiable, the properties of the associated Lipschitz homotopy groups listed in Theorem~\ref{lip_htpy_grps_contact_3_mflds} follow from similar tools and results in \cite{Dej} and \cite{Weg}.

\begin{theorem}\label{lip_htpy_grps_contact_3_mflds}
Let $(M,\xi)$ be a contact 3-manifold. Endow $(M,\xi)$ with a sub-Riemannian structure and consider the resulting \cc metric $\dcc{M}$. Then,
\begin{enumerate}
\item $\pilip{1}(M,\dcc{M})$ is uncountably generated, and
\item $\pilip{n}(M,\dcc{M})=0$ for $n\geq 2$.
\end{enumerate}
Furthermore, let $(M',\xi')$ be a contact 3-manifold which is endowed with a sub-Riemannian structure and let $\varphi:(M,\xi)\hookrightarrow(M',\xi')$ be an open distributional embedding.
\begin{enumerate} 
\setcounter{enumi}{2}
\item The homomorphism induced by $\varphi$ on first Lipschitz homotopy groups
\[
\varphi_\#:\pilip{1}(M,\dcc{M})\longrightarrow\pilip{1}(M',\dcc{M'})
\]
is injective.
\end{enumerate}
\end{theorem}

The paper is organized as follows. In section 2, we introduce necessary background on contact manifolds and sub-Riemannian manifolds. We then focus on distributional maps between sub-Riemannian manifolds. These are smooth maps whose derivative carries the distribution of the domain into the distribution of the codomain. We show that, with respect to the \cc metrics, any distributional map is locally Lipschitz. That every distributional embedding is locally biLipschitz is an immediate consequence, as is the biLipschitz Darboux theorem. Finally, we show that the unrectifiability of $\H^n$ implies that a contact $(2n+1)$-manifold is purely $k$-unrectifiable for $k>n$. 

In section 3, we recall the definition of Lipschitz homotopy groups. 
We then make use of a result of Wenger and Young (Theorem 5 in \cite{Weg}) that says that all Lipschitz maps from a Lipschitz simply connected, quasi-convex space into a purely 2-unrectifiable space factor through a metric tree. An immediate corollary is that all higher Lipschitz homotopy groups are trivial for purely 2-unrectifiable spaces \cite{Weg}. We also use the result of Wenger and Young to show that a distributional embedding of a purely 2-unrectifiable sub-Riemannian manifold into another induces an injective map on the associated first Lipschitz homotopy groups. Thus, since $\pilip{1}(\H^1)$ is uncountably generated \cite[Theorem 4.11 (2)]{Dej}, for any contact 3-manifold the first Lipschitz homotopy group is uncountably generated. Moreover, there is an uncountable subgroup of the first Lipschitz homotopy group for every connected open neighborhood of the chosen base point.

The work and results that follow are part of the author's PhD thesis \cite{dissertation}.

\medskip
\noindent {\bf Acknowledgment}. The author wishes to thank the referees for their time reviewing the paper and for their pointed feedback which was instrumental in improving the clarity and focus of the paper.  Additionally, the author wishes to thank David Ayala and Lukas Geyer for their feedback and guidance throughout the writing process. 

\nopagebreak
\section{Contact 3-manifolds are purely 2-unrectifiable.}



Any contact 3-manifold is locally modeled by a purely 2-unrectifiable space. Indeed, by the Theorem of Darboux, contact 3-manifolds locally look like copies of $\H^1$ \cite{Darb} and $\H^1$ is a purely 2-unrectifiable space \cite{Amb}. As will be shown, for any point in a contact 3-manifold, there is an open neighborhood of the point such that the distributional embedding guaranteed by Darboux restricted to the open neighborhood is a biLipschitz map with respect to the associated \cc metrics. Thus, these biLipschitz maps will carry this metric condition on $\H^1$ to the contact 3-manifold.  

In fact, we will show a more general result: since a contact $(2n+1)$-manifold is locally modeled by the $n$th Heisenberg group $\H^n$ (again by the Theorem of Darboux) and $\H^n$ is purely $k$-unrectifiable for all $k>n$ \cite{Mag04}, contact $(2n+1)$-manifolds are purely $k$-unrectifiable for $k>n$ when endowed with a \cc metric.

This result will be achieved by inspecting the interplay of distributional maps and the \cc lengths of paths. After covering some background material, we will show that the length of the image of a horizontal path under distributional map, which again is a horizontal path, is bounded. Thus, the distributional embedding guaranteed by the Theorem of Darboux can only distort lengths of paths, and thus distances between points, by a manageable amount. 

Next, we will account for subsets of contact manifolds not necessarily being geodesically convex. We will show that for any open ball in a sub-Riemannian manifold, there is a bounded open subset containing the ball in which distances between points in the ball can be well-approximated by horizontal paths that remain in the new bounded open subset. These tools will be enough to restrict a distributional embedding to a neighborhood such that the restriction is also biLipschitz with respect to the \cc metrics.




%

\subsection{Contact manifolds, horizontal paths, and distributional embeddings.}

\begin{convention}
Throughout this paper, the term ``manifold'' will refer to a smooth connected manifold, and the term ``distribution'' will refer to a smooth vector subbundle of a tangent bundle.
The tangent bundle of a manifold $M$ will be denoted $TM$. The derivative of a smooth map $f$ will be denoted $\de{f}$.
\end{convention}

\begin{definition}
A \emph{manifold with a distribution} is a pair $(M,\xi)$ composed of a manifold $M$ and a distribution $\xi\subset TM$. If the distribution is bracket-generating, the pair $(M,\xi)$ is called a \emph{Carnot manifold}. Additionally, if the manifold $M$ is of odd-dimension $2n+1$ and the bracket-generating distribution $\xi$ is co-dimension 1, the pair $(M,\xi)$ is called a \emph{contact $(2n+1)$-manifold}. 
\end{definition}


Contact manifolds are the primary interest of this paper. For a more thorough discussion of Carnot manifolds and bracket-generating distributions, see \cite{karmanova2008geometry}.  

The following example, called the $n$th Heisenberg group,  is the quintessential contact manifold in that all contact manifolds are locally modeled by a Heisenberg group.

\begin{example}[$n$th Heisenberg group]\label{H^n contact}
Let $M=\R^{2n+1}$ with coordinates denoted by $x_1,\ldots,x_n, y_1,\ldots,y_n,t$. Define a co-dimension 1 distribution on $\R^{2n+1}$ by 
\[
\xi^{std}\deq\span(X_1,\ldots,X_n,Y_1,\ldots,Y_n),
\]
where, for $i=1,\ldots,n$, the vector fields $X_i$ and $Y_i$ are defined by
\[
X_i\deq\basis{x_i}+2y_i\basis{t}\text{ and }Y_i\deq\basis{y_i}-2x_i\basis{t}.
\]
A calculation verifies that the $n$th Heisenberg group $\H^{n}\deq(\R^{2n+1},\xi^{std})$ is indeed a contact manifold.

%
\end{example}

We now describe means of probing the distributional structure of a manifold with a distribution. 

\begin{definition}
Let $(M,\xi)$ be a manifold with a distribution and let $N$ be a manifold with boundary. A smooth map $f:N\longrightarrow M$ is \emph{horizontal}, denoted $f:N\longrightarrow(M,\xi)$, if $\de{f}(TN)\subset \xi$. 
\begin{center}
\begin{tikzcd}
&& \xi \arrow[dd, hookrightarrow]  \\  \\

TN \arrow[rr, "\de{f}"]\arrow[dd] \arrow[rruu, dashed, "\exists"] && TM \arrow[dd] \\ \\

N \arrow[rr, "f"] && M.
\end{tikzcd}
\end{center}
If $N$ is a closed interval, the map $f$ is a \emph{horizontal path}. If the horizontal map is an embedding, the map $f$ is a \emph{horizontal embedding}.
\end{definition}

\begin{remark}
Legendrian knots are examples of horizontal embeddings of $\S^1$ into a contact 3-manifold.
\end{remark}

\begin{definition}\label{contacto}
Let $(M',\xi')$ and $(M,\xi)$ be manifolds with distributions. A smooth map $f:M'\longrightarrow M$ is a \emph{distributional map}, denoted $f:(M',\xi')\longrightarrow(M,\xi)$, if $\de{f}(\xi')\subset\xi$.
\begin{center}
\begin{tikzcd}
\xi' \arrow[dd, hookrightarrow] \arrow[rr, dashed, "\exists"] && \xi \arrow[dd, hookrightarrow]  \\ \\

TM' \arrow[rr, "\de{f}"]\arrow[dd] && TM \arrow[dd] \\ \\

M' \arrow[rr, "f"] && M.
\end{tikzcd}
\end{center}
If the distributional map is an embedding, the map $f$ is a \emph{distributional embedding}. If the distributional map is between contact manifolds, the map $f$ is a \emph{contact map}.
\end{definition}


Now that the contact structure on $\H^n$ has been established and we have a means of embedding this structure into other contact manifolds via distributional embeddings, we can make precise that contact manifolds are locally modeled by the Heisenberg group.


\begin{theorem}[Theorem of Darboux]\label{darboux} 
Let $(M,\xi)$ be a contact $(2n+1)$-manifold. For every $p\in M$, there exists an open distributional embedding 
\[
\p:\H^n\hookrightarrow(M,\xi)
\]
such that $\p(0)=p$.
\end{theorem}

The Theorem of Darboux was originally proved by Darboux in \cite{Darb}. For a modern statement and proof, see Theorem 2.24 in \cite{Gei}.

In the notation of Theorem~\ref{darboux}, the neighborhood $\p(\H^n)$, along with the associated distributional embedding $\p$, will be referred to as a \emph{Darboux neighborhood}.


\subsection{Sub-Riemannian manifolds and the Carnot-Carath\'{e}odory metric.}

Rather than probing the distribution of a contact manifold directly, we will endow each contact manifold with a metric sensitive to its distribution and then probe the resulting metric space. 

\begin{definition}
A \emph{sub-Riemannian manifold} is a triple $(M,\xi,g)$ consisting of a Carnot manifold $(M,\xi)$ and a smooth map
\[
g:\xi\times_M\xi\longrightarrow\R,
\]
such that, for each $p\in M$, the map $g:\xi_p \oplus \xi_p \rightarrow \R$ is an inner product on the vector space $\xi_p$. Such a map $g$ is referred to as a \emph{sub-Riemannian metric} on $(M,\xi)$.
\end{definition}

\begin{remark}
Any Carnot manifold can be endowed with a sub-Riemannian metric by restricting a Riemannian metric to the bracket-generating distribution. Going forward, we will assume that each Carnot manifold is endowed with a sub-Riemannian metric.  
\end{remark}


\begin{example}\label{H^n sub} 
Continuing Example~\ref{H^n contact}, the $n$th Heisenberg group $\H^n$ is naturally endowed with a sub-Riemannian metric. Indeed, define a sub-Riemann- ian metric $g$ such that, for each $p\in\H^n$, the vectors
\[
X_1(p),\ldots,X_n(p),Y_1(p),\ldots,Y_n(p)
\]
form an orthonormal basis for $\xi^{std}_p$. Going forward, it will be assumed that $\H^n$ has this sub-Riemannian metric.
\end{example}

To ensure that the metric imposed on a Carnot manifold is sensitive to the the distribution, the metric will be defined as a path metric where only the lengths of horizontal paths are considered.

\begin{definition}
Let $(M,\xi,g)$ be a sub-Riemannian manifold. The \emph{\cc length} of a horizontal path $\gamma:[a,b]\longrightarrow(M,\xi)$ is
\[
\length{M}(\gamma)\deq\int_a^b\sqrt{g(\dot{\gamma}(t),\dot{\gamma}(t))}~dt.
\]
%
%
%
%
The \emph{\cc metric} on $M$ is
\[
\dcc{M}(p,p')\deq\inf\Bigl\{\length{M}(\gamma)\mid\gamma:[a,b]\longrightarrow(M,\xi)\text{ with }\gamma(a)=p\text{ and }\gamma(b)=p'\Bigr\},
\]
for any $p,p'\in M$.
\end{definition}


\begin{convention}
Going forward, we will assume that any sub-Riemannian manifold is endowed with the \cc metric. The pair $(M,\dcc{M})$ will be used to identify the sub-Riemannian manifold $(M,\xi,g)$ as a metric space endowed with the \cc metric. The open ball centered at $p\in M$ of radius $R>0$ with respect to the \cc metric will be denoted by $\bcc{M}(p,R)$.
\end{convention}



Having endowed Carnot manifolds with a metric, we see how horizontal and distributional maps interact with the \cc metric. 
%

\begin{lemma}\label{bounds on lengths of paths}
Let $f:(M',\xi')\longrightarrow(M,\xi)$ be a distributional map between sub-Riemannian manifolds $(M',\xi',g')$ and $(M,\xi,g)$. Let $A\subset M'$ be a compact subset of $M'$. Then, there exists a value $B\geq 0$ such that, for any horizontal path $\gamma:[a,b]\longrightarrow(M',\xi')$ mapping into $A$, the following inequality holds:
\[
\length{M}(f\circ\gamma)\leq B~\length{M'}(\gamma).
\]
\end{lemma}

\begin{proof}

First, since the subset $A$ is compact, there exists a non-negative value $B\geq 0$ such that $||D_pf||\leq B$ for all elements $p\in A$.


Now, let $\gamma$ be an horizontal path in the manifold with distribution $(M',\xi')$ such that its image lies in the compact subset $A$. The composition $f\circ\gamma$ is a horizontal path in $(M,\xi)$. Then,
\begin{eqnarray*}
\length{M}(f\circ\gamma)
 & = & 
  \int_a^b\sqrt{g\left(\der{f}{\gamma(t)}(\dot{\gamma}(t)), \der{f}{\gamma(t)}(\dot{\gamma}(t))\right) }~dt \\
& \leq & 
 \int_a^b ||\der{f}{\gamma(t)}|| \sqrt{g'\left(\dot{\gamma}(t), \dot{\gamma}(t)\right) }~dt. \\
& \leq &   
  \int_a^b B \sqrt{g'\left(\dot{\gamma}(t), \dot{\gamma}(t)\right) }~dt \\					
					& = & 
					B~\length{M'}(\gamma).
\end{eqnarray*}

\end{proof}

We now describe means of probing the metric structure. Let $X$ and $Y$ be metric spaces with metrics $d^X$ and $d^Y$ respectively.

\begin{definition}
A map $f:X\longrightarrow Y$ is \emph{Lipschitz} if there exists $L\geq0$ such that for all $x,x'\in X$
\[
d^Y(f(x),f(x'))\leq L~d^X(x,x').
\]
Denote the set of all Lipschitz maps from $X$ to $Y$ by $\lipmap(X,Y).$ Furthermore, if $x_0\in X$ and $y_0\in Y$ are base points, the set of all based Lipschitz maps from $X$ to $Y$ is denoted by $\lipmap_{y_0}(X,Y).$
\end{definition}

\begin{definition}
A map $f:X\longrightarrow Y$ is \emph{locally Lipschitz} if, for all $p\in X$, there exists an open neighborhood $p\in U\subset X$ such that $f|_U$ is Lipschitz. 
\end{definition}


Lipschitz and locally Lipschitz maps are natural choices to substitute for smooth maps as smooth maps between Riemannian manifolds are locally Lipschitz with respect to the associated path metrics. As will be shown in Lemma~\ref{contactoembeddings are locally Lipscitz}, distributional maps between sub-Riemannian manifolds are locally Lipschitz with respect to the associated \cc metrics.

We now define a notion of an embedding between metric spaces. BiLipschitz bijections, for the purposes of this paper, are the appropriate notion of equivalence between metric spaces. 

\begin{definition}
A Lipschitz map $\p:X\longrightarrow Y$ is \emph{biLipschitz} if $\p$ is injective and its inverse map $\p^{-1}:\p(X)\longrightarrow X$ is also Lipschitz with respect to the metric $d^Y$ restricted to $\p(X)$. 
\end{definition}


\begin{definition}
A locally Lipschitz map $\p:X\longrightarrow Y$ is \emph{locally biLipschitz} if $\p$ is injective and for all $p\in X$, there exists an open neighborhood $p\in U\subset X$ such that $\p|_U$ is biLipschitz.
\end{definition}


\subsection{BiLipschitz Darboux theorem.}

As a consequence of the Theorem of Darboux, all contact $(2n+1)$-manifolds are locally modeled on $\H^n$, which, when endowed with the metric $\dcc{\H^n}$, is purely $k$-unrectifiable for $k>n$ \cite{Mag04}. In order to relay this metric quality on $\H^n$ to a contact $(2n+1)$-manifold $(M,\xi)$, we require an adjustment of the Theorem of Darboux. It will be shown in Corollary~\ref{biLipschitz Darboux} that the distributional embeddings of $\H^n$ into $(M,\xi)$ guaranteed by Darboux can be restricted such that the restrictions are biLipschitz with respect to the associated \cc metrics. 

The idea of utilizing the Theorem of Darboux to construct biLipschitz embeddings has appeared in the literature before. See for instance below Corollary 1.4 in \cite{LDOW14} or the proof of Proposition 4.1.2 in \cite{FLT19}.

It is worth noting that it is not immediate that the distributional embeddings guaranteed by Darboux are locally biLipschitz with respect to \cc metrics. The distributional embeddings are smooth and thus, assuming there are Riemannian metrics on the associated manifolds, locally biLipschitz with respect to the path metrics. But, these path metrics are not necessarily biLipschitz equivalent to the \cc metrics. Indeed, it is known for any sub-Riemannian manifold that these two metrics are not biLipschitz equivalent (Theorem 2.10 in \cite{Mon}).  

Thus, to guarantee that these distributional embeddings are taken to be locally biLipschitz, or even locally Lipschitz, we must better understand the \cc metric, in particular, where horizontal curves approximating the distance between two points live. As the \cc metric is defined in terms of lengths of horizontal curves, it is desirable to know how a distributional map can distort these lengths. Lemma~\ref{bounds on lengths of paths} is a tool for bounding lengths of paths that live in a given compact set. Choosing an open subset of the domain that is bounded then becomes the focus.

As Lemma~\ref{bounds on lengths of paths} yields a bound for horizontal paths that remain in a compact subset, it is important that the bounded open subset contains horizontal paths that well-approximate the \cc distance between some set of points. In practice, we cannot expect that the set containing the points and the set containing the horizontal paths to be equal. Given an arbitrary open subset of a contact manifold, it is unlikely that it is \emph{geodetically convex}, i.e., contains all length-minimizing horizontal paths between all of its points. Indeed, it is known that the only geodetically convex open subset of $\H^1$ is itself \cite{Monti}. 

So, we should not expect to be able to well-approximate Carnot-Carath\'{e}od- ory distance between points in a given bounded open subset via horizontal paths that remain in the open subset. Rather, given a bounded open subset, there is a larger but still bounded open subset of the ambient space in which the \cc distance between points in the former open subset can be well-approximated via horizontal paths that map into the latter.

\begin{lemma}\label{geodesic hull subset}
Let $(M,\xi,g)$ be a sub-Riemannian manifold. Consider the open ball $\bcc{M}(p,R)\subset M$. The set
\[
GH(p,R)\deq\bigcup_{q\in\bcc{M}(p,R)}\bcc{M}(q,2R)
\]
satisfies the following properties:
\begin{enumerate}
\item $GH(p,R)$ is an open subset of $M$, bounded with respect to $\dcc{M}$, and contains $\bcc{M}(p,R)$.
\item Let $x,y\in\bcc{M}(p,R)$ and let  $0<\varepsilon<2R-\dcc{M}(x,y)$. Then there exists a horizontal path 

\begin{center}
\begin{tikzcd}
{[0,1]} \arrow[rr, "\gamma_\varepsilon"] \arrow[dr] && (M,\xi)  \\ 

& (GH(p,R),\xi) \arrow[ur, hookrightarrow]&
\end{tikzcd}
\end{center}
from $x$ to $y$ such that
\[
\dcc{M}(x,y)<\length{M}(\gamma_\varepsilon)\leq\dcc{M}(x,y)+\varepsilon<2R.
\]
\end{enumerate}

\end{lemma}

Property \emph{(2)} is what is meant by \cc distance between two points being well-approximated via horizontal paths.

\begin{proof}

\emph{(1)} The subset $GH(p,R)$ is the union of open balls and is thus an open subset. From the definition, it is obvious that $GH(p,R)$ contains the open ball $\bcc{M}(p,R)$.

To see that $GH(p,R)$ is bounded, consider the distance between $p$ and an arbitrary element $x\in GH(p,R)$. By definition of $GH(p,R)$, there exists $q\in\bcc{M}(p,R)$ such that $x\in\bcc{M}(q,2R)$. By triangle inequality,
\[
\dcc{M}(x,p)\leq\dcc{M}(x,q)+\dcc{M}(q,p)<2R+R=3R
\]
and thus $x\in\bcc{M}(p,3R)$. So, $GH(p,R)$ is contained in the open ball $\bcc{M}(p,3R)$ and is therefore bounded. 



\emph{(2)} Let $0<\varepsilon<2R-\dcc{M}(x,y)$. By the infimum definition of $\dcc{M}$, there exists a horizontal path $\gamma_\varepsilon:[0,1]\longrightarrow(M,\xi)$ such that $\gamma_\varepsilon(0)=x$, $\gamma_\varepsilon(1)=y$, and
\[
\dcc{M}(x,y)<\length{M}(\gamma_\varepsilon)\leq\dcc{M}(x,y)+\varepsilon<2R.
\]
It remains to be verified that $\gamma_\varepsilon$ maps into $GH(p,R)$. It is sufficient to see that $\gamma_\varepsilon$ maps into $\bcc{M}(x,2R)\subset GH(p,R)$. 

Take $t\in[0,1]$. The restriction $\gamma_\varepsilon|_{[0,t]}$ is a horizontal path in $(M,\xi)$ connecting $x$ and $\gamma_\varepsilon(t)$. By the infimum definition of $\dcc{M}$, the \cc distance between $x$ and $\gamma_\varepsilon(t)$ is no more than the length of this restriction; 
\[
\dcc{M}(x,\gamma_\varepsilon(t))\leq\length{M}(\gamma_\varepsilon|_{[0,t]}). 
\]

Obviously, the length of $\gamma_\varepsilon|_{[0,t]}$ is no more than the length of $\gamma_\varepsilon$. Thus the following inequality holds,
\[
\dcc{M}(x,\gamma_\varepsilon(t))\leq\length{M}(\gamma_\varepsilon|_{[0,t]})\leq\length{M}(\gamma_\varepsilon)\leq\dcc{M}(x,y)+\varepsilon<2R
\]
and therefore $\gamma_\varepsilon(t)\in\bcc{M}(x,2R)$.

\end{proof}

So, given an open ball with respect to the \cc metric, there is a bounded open subset that contains all horizontal paths that well-approximate the \cc distance between points in the ball. We will use the bound guaranteed by Lemma~\ref{bounds on lengths of paths} on this larger bounded set to guarantee that distributional maps are locally Lipschitz.

{ 

\begin{lemma}\label{contactoembeddings are locally Lipscitz}
Let $(M',\xi',g')$ and $(M,\xi,g)$ be sub-Riemannian manifolds and let $\p:(M',\xi')\longrightarrow(M,\xi)$ be a distributional map. Then, the map 
\[
\p:(M',\dcc{M'})\longrightarrow(M,\dcc{M})
\]
is locally Lipschitz with respect to the \cc metrics.
\end{lemma}



\begin{proof}

Fix a point $p\in M'$. Take a radius $R>0$ such that the closed ball $\closure{\bcc{M'}(p,R)}$ is compact. Let $V'=\bcc{M'}(p,R/4)$.

Let $q,q'\in V'$. By Lemma~\ref{geodesic hull subset}, there exists an open and bounded subset $GH(p,R/4)\subset M'$ containing $V'$ in which the \cc distance between $q$ and $q'$ can be well-approximated by lengths of horizontal paths in $GH(p,R/4)$. Also, the subset $GH(p,R/4)\subset\bcc{M'}(p,3R/4)\subset \bcc{M'}(p,R)$ is contained in the compact subset $\closure{\bcc{M'}(p,R)}$.


Let $\varepsilon>0$ be given. Then, there exists a horizontal path $\gamma_\varepsilon$ contained in the open subset $GH(p,R/4)$ connecting $q$ and $q'$ such that
\[
\length{M'}(\gamma_\varepsilon)\leq\dcc{M'}(q,q')+\varepsilon.
\]
Now, the open subset $GH(p,R/4)$ is contained in the compact subset $\closure{\bcc{M'}(p,R)}$. By Lemma~\ref{bounds on lengths of paths}, for the distributional map $\p$, there exists a value $B\geq0$ independent of $\varepsilon$ and $\gamma_\varepsilon$ such that 
\[
\length{M}(\p\circ\gamma_\varepsilon)\leq B~\length{M'}(\gamma_\varepsilon).
\]


Since $\p\circ\gamma_\varepsilon$ is a horizontal path in $M$ connecting the points $\p(q)$ and $\p(q')$, by the infimum definition of the metric, 
\[
\dcc{M}(\p(q),\p(q'))\leq\length{M}(\p\circ\gamma_\varepsilon).
\]
Stringing these inequalities together, we get the following:
\[
\dcc{M}(\p(q),\p(q'))\leq\length{M}(\p\circ\gamma_\varepsilon)\leq B~\length{M'}(\gamma_\varepsilon)\leq B~(\dcc{M'}(q,q')+\varepsilon).
\]
As $\varepsilon$ can be taken to be arbitrarily small, $\dcc{M}(\p(q),\p(q'))\leq B~\dcc{M'}(q,q')$. Therefore, the map $\p$ is Lipschitz on the neighborhood $V'$ of the point $p$. Since $p$ was arbitrary, $\p$ is locally Lipschitz.

\end{proof}

\begin{remark}\label{remark_horizontal_is_lipschitz} 
Since Riemannian manifolds are sub-Riemannian manifolds where the distribution is taken to be the entire tangent bundle, Lemma~\ref{contactoembeddings are locally Lipscitz} also implies that smooth maps between Riemannian manifolds and horizontal maps from a Riemannian manifold into a sub-Riemannian manifold are locally Lipschitz with respect to the associated path metrics. If in addition a horizontal map is an embedding with a compact domain, the mapping is biLipschitz \cite[Theorem 3.1]{Dej}.
\end{remark}

This strategy can be used to guarantee that any distributional embedding is locally biLipschitz with respect to the \cc metrics. On the image of such a distributional embedding, the inverse map is also a distributional map. Indeed, as will be shown in the following argument, Lemma~\ref{contactoembeddings are locally Lipscitz} yields that this inverse map is locally Lipschitz and thus the original map is locally biLipschitz.

\begin{lemma}\label{contactoembeddings are locally bilipschitz}
Let $(M',\xi',g')$ and $(M,\xi,g)$ be sub-Riemannian manifolds and let $\p:(M',\xi')\hookrightarrow(M,\xi)$ be an open distributional embedding. Then the map 
\[
\p:(M',\dcc{M'})\hookrightarrow(M,\dcc{M})
\]
is locally biLipschitz with respect to the associated \cc metrics. 
\end{lemma}



\begin{proof}
Let $p\in M'$ be a point in $M'$. By Lemma~\ref{contactoembeddings are locally Lipscitz}, there exists an open neighborhood $p\in V'\subset M'$ such that the restriction $\p|_{V'}$ is Lipschitz.

Now, since the map $\p$ is invertible on its image, 
\[
\p^{-1}|_{\p(V')}:(\p(V'),\xi|_{\p(V')})\hookrightarrow (V',\xi|_{V'})
\]
is also a distributional map. 
Again by Lemma~\ref{contactoembeddings are locally Lipscitz}, there exists an open neighborhood $\p(p)\in U\subset\p(V')$ on which $\p^{-1}$ is Lipschitz. Then, $V=\p^{-1}(U)\subset V'$ is an open subset on which $\p|_{V}$ is Lipschitz and invertible. Since the inverse is also Lipschitz, 
\[
\p|_{V}:(V,\dcc{M'})\stackrel{\cong}{\longrightarrow}(\p(V),\dcc{M})
\]
is biLipschitz.

\end{proof}

With this more general result established, a biLipschitz Darboux Theorem is an immediate corollary. 

\begin{corollary}[BiLipschitz Theorem of Darboux]\label{biLipschitz Darboux} 
Let $(M,\xi)$ be a contact $(2n+1)$-manifold. For every $p\in M$, 
there exists an locally biLipschitz open distributional embedding 
\[
\p:\H^n\hookrightarrow(M,\dcc{M})
\]
such that $\p(0)=p$. 
\end{corollary}

\begin{proof}
By the Theorem of Darboux (Theorem~\ref{darboux}), for any point $p\in M$, there exists an open distributional embedding of $\H^n$, 
\[
\p:\H^n\hookrightarrow(M,\xi)
\]
where $\p(0)=p$. By Lemma~\ref{contactoembeddings are locally bilipschitz}, the embedding $\p$ is locally biLipschitz.
\end{proof}

For such a biLipschitz map $\p$ guaranteed by Corollary~\ref{biLipschitz Darboux}, there is an open neighborhood $p\in V\subset\H^n$ such that the restriction $\p|_V$ is biLipschitz. Such a neighborhood $\p(V)$, along with the associated biLipschitz distributional embedding, will be referred to as a \emph{biLipschitz Darboux neighborhood}.



\subsection{Unrectifiability of contact manifolds.}

\begin{convention}
Here and going forward, $\R^k$, and any subset thereof, is endowed with the Lebesgue measure $\mathcal{L}^k$ and the Euclidean metric $d^\delta$ unless otherwise mentioned.
\end{convention}

\begin{definition}
Let $k\geq 1$ be a positive integer. A metric space $X$ is \emph{purely $k$-unrectifiable}  if, for all Borel sets $A\subset\R^k$ and all Lipschitz maps $f:A\longrightarrow X$, the $k$-dimensional Hausdorff measure of the image vanishes: 
\[
\hm{k}(\Ima(f))=0.
\]
\end{definition}

Informally, a space is purely $k$-unrectifiable if Lipschitz maps from $k$-dimensi-\\ onal Euclidean space into the space cannot sweep out any of the $k$-dimensional Hausdorff measure. 



The following theorem was first proven by Ambrosio and Kirchheim in the case $n=1$ \cite[Theorem 7.2]{Amb}. A more general result implying this theorem was shown by Magnani \cite{Mag04}. Also see \cite[Theorem 1.1]{Wildrick}. 

\begin{theorem}\label{heisenberg_purely_2}
The $n$th Heisenberg group $\H^n$ is purely $k$-unrectifiable for $k>n$.
\end{theorem}

As contact $(2n+1)$-manifolds as metric spaces are locally modeled by $\H^n$, such spaces are the union of biLipschitz Darboux neighborhoods.  Since $\H^n$ is purely $k$-unrectifiable for $k>n$, we will show that the union of the biLipschitz Darboux neighborhoods is as well.

\begin{theorem}\label{contact 3-mflds purely unrectifiable}
Any contact $(2n+1)$-manifold $(M,\xi)$, endowed with the \cc metric, is purely $k$-unrectifiable for $k>n$. 
\end{theorem}

\begin{proof}

Fix a positive integer $k>n$. Construct a cover of $(M,\xi)$ by biLipschitz Darboux neighborhoods. By Corollary~\ref{biLipschitz Darboux}, each point in $(M,\xi)$ has a biLipschitz Darboux neighborhood. $M$ can be covered by such neighborhoods and, since $M$ is a manifold, it can be reduced to a countable cover. 

Let $\{\p_\alpha:(V_\alpha,\dcc{\H^n})\hookrightarrow(M,\dcc{M})\}_{\alpha\in J}$ denote a countable collection of biLipschitz open distributional embeddings where $V_\alpha\subset\H^n$ is open for each $\alpha\in J$, such that $\{\p_\alpha(V_\alpha)\}_{\alpha\in J}$ is a countable cover of $M$.

Let $f:(A,d^\delta)\longrightarrow (M,\dcc{M})$ be a Lipschitz map whose domain $A\subset\R^k$ is a Borel set. To verify that $(M,\dcc{M})$ is purely $k$-unrectifiable, it is enough to show that $\hm{k}(\Ima f)=0$.

Fix an $\alpha\in J$ and consider $f$ restricted to the relatively open subset $f^{-1}(\p_\alpha(V_\alpha))\subset A$. By Corollary~\ref{biLipschitz Darboux}, $\p_\alpha^{-1}:(\p_\alpha(V_\alpha),\dcc{M})\longrightarrow (V_\alpha,\dcc{\H^n})$ is a Lipschitz map. As $f|_{f^{-1}(\p_\alpha(V_\alpha))}$ maps into $\p_\alpha(V_\alpha)$, 
\[
\p^{-1}_\alpha\circ f|_{f^{-1}(\p_\alpha(V_\alpha))}:(f^{-1}(\p_\alpha(V_\alpha)),d^\delta)\longrightarrow (V_\alpha, \dcc{\H^n})
\]
is defined and is Lipschitz;
\begin{center}
\begin{tikzcd}
A\arrow[rr, "f"] && M \\ \\

&& \p_\alpha(V_\alpha) \arrow[uu, hookrightarrow] \\ \\

f^{-1}(\p_\alpha(V_\alpha))\arrow[uuuu,hookrightarrow] \arrow[rr, "\p_\alpha^{-1}\circ f|"]\arrow[rruu, "f|_{f^{-1}(\p_\alpha(V_\alpha))}" near end] && V_\alpha. \arrow[uu, "\cong", "\p_\alpha"']
\end{tikzcd}
\end{center}
As $\p_\alpha(V_\alpha)\subset M$ is open and $f$ is continuous, $f^{-1}(\p_\alpha(V_\alpha))\subset A$ is an open subset of a Borel set and is thus Borel. 

Since $\H^n$ is purely $k$-unrectifiable (Theorem~\ref{heisenberg_purely_2}), 
\[
\hm{k}(\Ima(\p^{-1}_\alpha\circ f|_{f^{-1}(\p_\alpha(V_\alpha))}))=0.
\]
As $\p_\alpha$ is Lipschitz and the Lipschitz image of a $\hm{k}$-measure zero set is a $\hm{k}$-measure zero set,
\[
\hm{k}(\Ima(f|_{f^{-1}(\p_\alpha(V_\alpha))}))=\hm{k}(\p_\alpha(\Ima(\p^{-1}_\alpha\circ f|_{f^{-1}(\p_\alpha(V_\alpha))})))=0.
\]


Now, note that $\Ima f=\ds\bigcup_{\alpha\in J}\Ima(f|_{f^{-1}(\p_\alpha(V_\alpha))})$. By subadditivity of the outer measure $\hm{k}$,
\[
0\leq\hm{k}(\Ima f)\leq\sum_{\alpha\in J}\hm{k}(\Ima(f|_{f^{-1}(\p_\alpha(V_\alpha))})).
\]
Since $\alpha\in J$ above was arbitrary, the right hand side of the inequality is zero and $\hm{k}(\Ima f)=0$.
\end{proof}

%
%
%
%
%
%
%
%
%
%
%
%

\nopagebreak
\section{Lipschitz homotopy groups of purely 2-unrectifiable sub-Riemannian manifolds.}

\subsection{Lipschitz homotopy groups.}

Having endowed Carnot manifolds with a \cc metric, we report the probing of the metric structure by Lipschitz maps via Lipschitz homotopy groups.  Going forward, let $I=[0,1]$ denote the unit interval.

\begin{definition}\label{lip homotopy groups}
Let $s_0\in\S^n$ be a base point for the $n$-sphere. For a based metric space $(X,d)$ with basepoint $x_0\in X$, the \emph{$n$th Lipschitz homotopy group} is
\[
\pilip{n}((X,d),x_0)\deq \lipmap_{x_0}(\S^n,X)/\sim,
\]
where two based Lipschitz maps in $f_0,f_1\in\lipmap_{x_0}(\S^n,X)$ are equivalent, $f_0\sim f_1$, if there exists a Lipschitz homotopy $H\in\lipmap(I\times\S^n,X)$ such that 
\begin{eqnarray*}
H|_{\{0\}\times\S^n}=f_0, & & \\
H|_{\{1\}\times\S^n}=f_1, &\text{ and }& \\
H|_{I\times\{s_0\}}=x_0. & &
\end{eqnarray*}
\end{definition}

This definition agrees with the definition of Lipschitz homotopy groups provided in Definition 4.1 of \cite{Dej}. Provided that $X$ is a Riemannian manifold with the associated path metric, the Lipschitz homotopy groups of $X$ agree with the classical homotopy groups (Theorem 4.3 in \cite{Dej}).

The base point will often be suppressed when it is not of utmost importance. In fact for any sub-Riemannian manifold, by Chow-Rashevskii theorem, the $n$th Lipschitz homotopy group is the same no matter the choice of base point. See Theorem 4.2~(2) in \cite{Dej}.

\subsection{A distributional open embedding induces an injective map on $\pilip{1}$.}


%

In \cite{Weg}, Wenger and Young showed that certain Lipschitz maps into a purely 2-unrectifiable space factor through metric trees. 

\begin{theorem}[Theorem 5 in \cite{Weg}]\label{Wenger and Young}
Let $X$ be a quasi-convex metric space with $\pilip{1}(X)=0$. Let furthermore $Y$ be a purely 2-unrectifiable metric space. Then every Lipschitz map from $X$ to $Y$ factors through a metric tree.
\end{theorem} 

Since metric trees are Lipschitz contractible, any Lipschitz map with appropriate domain and purely 2-unrectifiable target is Lipschitz null-homotopic. For example, Corollary~\ref{W_Y Corollary} covers the case that the domain is an $n$-sphere with $n\geq 2$. This result is stated in \cite{Weg} as a corollary to Theorem~\ref{Wenger and Young}. Theorem~\ref{lip_htpy_grps_contact_3_mflds} (2) then follows immediately.


\begin{corollary}[\cite{Weg}]\label{W_Y Corollary}
Let $Y$ be a purely 2-unrectifiable metric space. If $n\geq2$ and $\alpha:\S^n\longrightarrow Y$ is a Lipschitz map, then $\alpha$ is Lipschitz null-homotopic. That is, $\pilip{n}(Y)=0$.
\end{corollary}

\begin{proof}
The $n$-sphere $\S^n$, with its standard Riemannian metric, is quasi-convex and is Lipschitz simply connected. 
Thus, by Theorem~\ref{Wenger and Young}, the Lipschitz map $\alpha$ factors through a metric tree $T$,
\begin{center}
\begin{tikzcd}
\S^n\arrow[rr, "\alpha"] \arrow[dr, "\psi"'] && Y.  \\ 

& T \arrow[ur, "\phi"']&
\end{tikzcd}
\end{center}
The maps $\psi:\S^n\longrightarrow T$ and $\phi:T\longrightarrow Y$ are Lipschitz as well. 


Since $T$ is a metric tree, $T$ is contractible by a Lipschitz homotopy $h:I\times T\longrightarrow T$. Therefore, the homotopy  $H:I\times\S^n\longrightarrow Y$ given by $H(p,t)\deq\phi(h(\psi(p),t))$ is a Lipschitz null-homotopy of the map $\alpha$.


\end{proof}

\begin{proof}[Proof of Theorem~\ref{lip_htpy_grps_contact_3_mflds} (2)] 
Since the contact 3-manifold is purely 2-unrectifiable (Theorem~\ref{contact 3-mflds purely unrectifiable}), the result follows immediately from Corollary~\ref{W_Y Corollary}.
\end{proof}

{ 

In the remainder of this paper, we will apply Theorem~\ref{Wenger and Young} to argue that an open distributional embedding of a purely 2-unrectifiable sub-Riemannian manifold into another induces an injective homomorphism between the respective first Lipschitz homotopy groups. 
Parts (3) and (1) of Theorem~\ref{lip_htpy_grps_contact_3_mflds} will follow immediately. 

First, we will show that the Lipschitz null homotopy of a Lipschitz null homotopic loop can be taken such that the loop shrinks to a point along its image. 

Going forward, let $\D^{2}$ denote the unit ball in $\R^{2}$. Also, by \emph{subtree} we will mean a nonempty, connected, compact subset of a metric tree. A subtree is then a metric tree as well.

\begin{lemma}\label{homotopy of a disk takes image on boundary}
Let $Y$ be a purely 2-unrectifiable metric space. Let $H_0:\D^2\rightarrow Y$ be a Lipschitz map. Then, the map $H_0$ is Lipschitz homotopic to a Lipschitz map $H_1:\D^2\rightarrow Y$ such that the image of $H_1$ is contained in the image of $H_0$ restricted to the boundary of the 2-disk: $H_1\left(\D^2\right)\subset H_0\left(\partial\D^2\right)$. Furthermore, the homotopy is relative to the boundary of $\D^2$. 
\end{lemma}


\begin{proof}
Let $H_0:\D^2\longrightarrow Y$ be a Lipschitz map. By Theorem~\ref{Wenger and Young}, since $\D^2$ is quasi-convex and Lipschitz simply-connected, the map $H_0$ factors through a metric tree $T$:
\begin{center}
\begin{tikzcd}
\D^2 \arrow[rr, "H_0"] \arrow[rd, "\psi"'] & & Y. \\ 
& T \arrow[ru, "\phi"'] & 
\end{tikzcd}
\end{center}

Both maps $\psi$ and $\phi$ are Lipschitz. Since $\psi$ is then continuous, the image  $\psi(\partial\D^2)\subset T$ is connected and compact. So, $\psi(\partial\D^2)$ is a subtree of the metric tree $T$. Thus, there exists a Lipschitz deformation retract $F:I\times T\rightarrow T$ of the tree $T$ onto the subtree $\psi(\partial\D^2)$. 

Consider the Lipschitz map 
\[
\phi\circ F\circ (\id_I\times\psi):I\times\D^2\rightarrow Y.
\]
We will argue that this map is a homotopy between $H_0$ and 
\[
H_1:=\phi\circ F\circ (\id_I\times\psi)\circ(1\times\id_{\D^2}):\D^2\longrightarrow Y
\] 
satisfying the desired properties.

Since the map $F$ is a deformation retract of the metric  tree $T$, we have an equality of maps $F\circ (0\times\id_T)=\id_T$. Thus, precomposing $\phi\circ F\circ (\id_I\times\psi)$ by the natural inclusion of $\{0\}\times\D^2$ into $I\times\D^2$ yields the original map $H_0=\phi\circ\psi$. This equality is indicated in the filled diagram in Figure~\ref{diagram}.

\begin{center}
\begin{figure}[h]
\begin{tikzcd}
\{1\}\times\D^2 \arrow[dd, hookrightarrow] \arrow[r, dashed, "{1_I\times\psi}"] & \{1\}\times T \arrow[dd, hookrightarrow] \arrow[r, twoheadrightarrow, dashed] & {\psi(\partial\D^2)} \arrow[dd, hookrightarrow]\arrow[r, dashed, "{\phi|}"] & {\phi\circ\psi(\partial\D^2)}\arrow[dd, hookrightarrow]\arrow[r, equal] & {H_0\left(\partial\D^2\right)} \arrow[ddl, hookrightarrow] \\ \\
I\times \D^2 \arrow[r, "{\id_I\times\psi}"] & I\times T \arrow[r, "F"] & T \arrow[r, "\phi"]& Y. & \\ \\
\{0\}\times\D^2 \arrow[uu, hookrightarrow] \arrow[uurrr, bend right=20, "{H_0}"]
\end{tikzcd}
\caption{}
\label{diagram}
\end{figure}
\end{center}

We now argue that there are factorizations of the maps $\id_I \times\psi$, $F$, and $\phi$ as is indicated by the dashed arrows in Figure~\ref{diagram}. 

For the Lipschitz map $\id_I\times\psi:I\times\D^2\rightarrow I\times T$, precomposing by the natural inclusion of $\{1\}\times\D^2$ yields the Lipschitz map $1_I\times\psi:\{1\}\times\D^2\longrightarrow\{1\}\times T,$ where $1_I:\{1\}\rightarrow\{1\}$ is a constant map. Next, since the map $F$ is a deformation retract of the metric tree $T$ onto the subtree $\psi\left(\partial\D^2\right)$, we have that $F\circ (1\times\id_T):\{1\}\times T\twoheadrightarrow\psi\left(\partial\D^2\right)$ maps onto the subtree $\psi\left(\partial\D^2\right)$.The third dashed arrow comes from restricting the Lipschitz map $\phi$ to the subset $\psi(\partial\D^2)\subset T$. Finally, since $H_0$ factors into the composition $\phi\circ\psi$, when restricted to the boundary of the 2-disk there is an equality of sets $\phi\circ\psi\left(\partial\D^2\right)=H_0\left(\partial\D^2\right)$.

Therefore, precomposing the map $\phi\circ F\circ (\id_I\times\psi)$ by the natural inclusion of $\{1\}\times\D^2$ into $I\times\D^2$ yields a map $H_1$ that has image contained in $H_0\left(\partial\D^2\right)\subset Y$. Moreover, since the map $F$ is a deformation retract onto $\psi(\partial\D^2)$, the Lipschitz homotopy $\phi\circ F\circ(\id_I\times\psi)$ is constant on $\partial\D^2$ for all time $t\in I$.

\end{proof}

We now show that each open distributional embedding from a purely 2-unrectifiable sub-Riemannian manifold into another induces an injective map on their respective first Lipschitz homotopy groups. 

Before proceeding, note that an open distributional embedding $\varphi:(M,\xi)\hookrightarrow(M',\xi')$ between sub-Riemannian manifolds does induce a homomorphism between Lipschitz homotopy groups. Indeed, via Lemma~\ref{contactoembeddings are locally Lipscitz}, the map $\p$ is locally Lipschitz and thus, for any Lipschitz map $\alpha:\S^n\longrightarrow(M,\dcc{M})$, the map $\p\circ\alpha$ is Lipschitz since its domain is compact.

\begin{theorem}\label{inclusion yields injective map of homotopy groups} 
Let $(M,\xi,g)$ and $(M',\xi',g')$ be purely 2-unrectifiable sub-Riema- nnian manifolds. Let $\varphi:(M,\xi)\hookrightarrow(M',\xi')$ be an open distributional embedding. Then the homomorphism induced by $\varphi$ on first Lipschitz homotopy groups
\[
\varphi_\#:\pilip{1}(M,\dcc{M})\longrightarrow\pilip{1}(M',\dcc{M'})
\]
is injective.
\end{theorem}


\begin{proof}
As $\varphi_\#$ is a homomorphism, we can show that the map is injective by showing that the kernel of the map is trivial. 

Let $\alpha:\S^1\rightarrow(M,\dcc{M})$ be a Lipschitz map that represents an element of the kernel of $\varphi_\#$. So, there exists a Lipschitz map $H:\D^2\rightarrow(M',\dcc{M'})$ such that $H$ restricted to the boundary is the Lipschitz map $\varphi\circ\alpha$:
\[
H|_{\partial\D^2}=\varphi\circ\alpha.
\]
Since $\varphi\circ\alpha:\S^1\rightarrow(M',\dcc{M'})$ is the composition of Lipschitz functions, $\varphi\circ\alpha$ is Lipschitz. 

By Lemma~\ref{homotopy of a disk takes image on boundary}, the Lipschitz homotopy $H$ can be taken such that the image of $H$ is contained in the image of the Lipschitz map $\varphi\circ\alpha$. Thus, $H$ takes image entirely in the image of $\varphi$:
\[
\Ima(H)\subset\Ima(\varphi\circ\alpha)\subset\Ima(\varphi).
\]
Since the inverse $\varphi^{-1}:\Ima{\varphi}\rightarrow (M,\xi)$ is a distributional diffeomorphism, the map given by composition
\[
\varphi^{-1}\circ H:\D^2\longrightarrow (M,\dcc{M})
\]
is Lipschitz and, when the map is restricted to the boundary of $\D^2$ equals the map $\alpha$. Thus, $\alpha$ is Lipschitz null homotopic. Therefore, the only element in the kernel of $\varphi_\#$ is the trivial homotopy class.

\end{proof}

\begin{proof}[Proof of Theorem~\ref{lip_htpy_grps_contact_3_mflds} (3)] 
Since contact 3-manifolds are purely 2-unrectifiable (Theorem~\ref{contact 3-mflds purely unrectifiable}), the result follows immediately from Theorem~\ref{inclusion yields injective map of homotopy groups}.
\end{proof}

\begin{proof}[Proof of Theorem~\ref{lip_htpy_grps_contact_3_mflds} (1)]
By the Theorem of Darboux (Theorem~\ref{darboux}), there is an open distributional embedding of $\H^1$ into $(M,\xi)$. By Theorem~\ref{lip_htpy_grps_contact_3_mflds} (3), the embedding induces an injective map between the associated Lipschitz homotopy groups. Thus, as the group $\pilip{1}(\H^1)$ is uncountably generated (Theorem 4.11 (2) in \cite{Dej}), the group $\pilip{1}(M,\dcc{M})$ is also uncountably generated.
\end{proof}

%

\begin{remark}
Theorem~\ref{inclusion yields injective map of homotopy groups} indicates that the cardinality of $\pilip{1}(M,\dcc{M})$ is extremely large for any purely 2-unrectifiable sub-Riemannian manifold $(M,\xi,g)$. For a base point in $M$, any connected, open neighborhood $(U,\xi|_U)$ is a purely 2-unrectifiable sub-Riemannian manifold that openly and distributionally embeds into $(M,\xi)$. Thus, a copy of the set $\pilip{1}(U,\dcc{U})$ is a subgroup of $\pilip{1}(M,\dcc{M})$. Additionally, from Theorem~\ref{lip_htpy_grps_contact_3_mflds} (1), if $(M,\xi)$ is a contact 3-manifold, the  subgroup $\pilip{1}(U,\dcc{U})$ in $\pilip{1}(M,\dcc{M})$ is of uncountable cardinality. 
\end{remark}
}

\bibliography{bib}{}
\bibliographystyle{plain}

\end{document}